\documentclass[11pt, reqno]{amsart}
\usepackage[left=3cm,right=3cm,marginparwidth=2cm]{geometry}
\usepackage[T1]{fontenc}
\usepackage{amsmath}
\usepackage{amsthm}
\usepackage{amssymb}
\usepackage{mathtools}
\usepackage{accents}
\usepackage{cases}
\usepackage{empheq}
\usepackage{pgfplots}
\usepgfplotslibrary{groupplots}
\pgfplotsset{compat=1.14}
\usepackage{fullpage}
\usepackage{subcaption}
\usepackage{graphicx}

\mathtoolsset{showonlyrefs}

\newcounter{custom}
\numberwithin{custom}{section} 

\let\oldsubsection\subsection
\renewcommand{\subsection}{\stepcounter{custom}\oldsubsection}
\newtheorem{theorem}[custom]{Theorem}

\newtheorem{lemma}[custom]{Lemma}

\newtheorem{remark}[custom]{Remark}
\newtheorem{definition}[custom]{Definition}

\usepackage{tkz-graph}
\usepackage{tikz}
\usetikzlibrary{arrows}
\usepackage{algorithm}
\usepackage[noend]{algpseudocode}
\usepackage{wrapfig}
\usepackage{caption}

\usepackage{tristan}
\usepackage{tikz-3dplot}

\renewcommand{\vec}[1]{\ensuremath{\boldsymbol{#1}}}

\usepackage{amsaddr}

\title[]{Cycle-Free Polytopal Mesh Sweeping for Boltzmann Transport}
\author[Calloo, A. et al.]{Ansar Calloo$^{1}$, Matthew Evans $^{2*}$, Henry Lockyer $^{2}$, Fran\c{c}ois Madiot$^{3}$, Tristan Pryer$^{2,4}$, Luca Zanetti$^{2,4}$}
\address{$^1$Universit\'e Paris-Saclay, CEA, Service de G\'enie Logiciel pour la Simulation, 91191, Gif-sur-Yvette, France}
\address{$^2$Mathematical Sciences, University of Bath}
\address{$^3$Universit\'e Paris-Saclay, CEA, Service d'\'Etudes des R\'eacteurs et de Math\'ematiques Appliqu\'ees, 91191, Gif-sur-Yvette, France}
\address{$^4$Institute of Mathematical Innovation, University of Bath}
\email{$^*$\url{mje45@bath.ac.uk}}
\date{}
\begin{document}

\bibliographystyle{acm}

\begin{abstract}
  We introduce a novel property of bounded Voronoi tessellations that
  enables cycle-free mesh sweeping algorithms. We prove that a
  topological sort of the dual graph of any Voronoi tessellation is
  feasible in any flow direction and dimension, allowing
  straightforward application to discontinuous Galerkin (DG)
  discretisations of first-order hyperbolic partial differential
  equations and the Boltzmann Transport Equation (BTE) without
  requiring flux-cycle corrections.

  We also present an efficient algorithm to perform the topological
  sort on the dual mesh nodes, ensuring a valid sweep ordering. This
  result expands the applicability of DG methods for transport
  problems on polytopal meshes by providing a robust framework for
  scalable, parallelised solutions. To illustrate its effectiveness,
  we conduct a series of computational experiments showcasing a DG
  scheme for BTE, demonstrating both computational efficiency and
  adaptability to complex geometries.
\end{abstract}

\maketitle

\section{Introduction}
\label{section:introduction}

\subsection{Motivation}

The Boltzmann transport equation (BTE) is an integro-differential
equation initially developed to model particle distributions in
non-equilibrium thermodynamic systems \cite{Cercignani:1988}. Today,
the BTE is central to numerous applications, particularly in radiation
transport and neutronics, where it underpins simulations in nuclear
reactor design, including modern molten salt reactors
\cite{BetzlerPowersWorrall:2017}, as well as radiotherapy techniques
\cite{CoxHattamKyprianouPryer:2023,CrossleyHabermannHortonKoskelaKyprianouOsman:2024}. These
applications rely on robust numerical methods to address the
complexity of the BTE, which arises from its integro-differential
structure and high dimensionality. Despite the theoretical challenges
associated with proving the existence and uniqueness of solutions,
numerical methods for the BTE have been a focus of research for over a
century \cite[c.f]{Hilbert:1904,LewisMiller:1984}.

The discrete ordinates method (DOM) is one of the primary techniques
for approximating solutions to the BTE in radiation transport
\cite{chandrasekhar1960radiative}. DOM discretises the angular
variables to approximate the directions of particle flow and, in the
mono-energetic case, yields a spatial-temporal equation that can be
formulated into a large linear system. For certain spatial
discretisations, this system can be transformed into a lower
triangular form, which enables an efficient mesh-sweeping solution
\cite{AzmySartoriLarsenMorel:2010}. The sweeping process itself
involves solving the triangular system in a sequential, directional
manner, where each step depends only on previously solved values.

However, not all mesh configurations support this lower triangular
transformation. Structured, regular grids are typically amenable to
such sweeping techniques, but unstructured or irregular meshes often
introduce cyclic dependencies that prevent a straightforward
transformation, thereby obstructing the sweep \cite{plimpton2005}.

Furthermore, fine discretisations generate substantial numbers of
unknowns, making parallelisation essential to manage computational
load and reduce solution times. Achieving efficient parallelisation,
however, depends heavily on mesh compatibility with sweep ordering, as
incompatible meshes can significantly hinder parallel efficiency. This
challenge motivates the development of mesh types and algorithms
specifically designed to allow sweeping on complex geometries in a
parallelised setting, maximising computational performance on modern
architectures
\cite{AdamsAdamsHawkinsSmithRauchwergerAmatoBaileyFalgout:2013,Adams_2020}.

In recent years, discontinuous Galerkin (DG) methods on polygonal
meshes have garnered significant attention due to their flexibility in
handling complex geometries, high-order accuracy and their suitability
for adaptive mesh refinement, which is important for resolving
localised features within transport problems
\cite{CangianiGeorgoulisHouston:2014}. Additionally, DG methods on
polygonal meshes allow for more efficient mesh generation in irregular
domains and can reduce computational costs by minimising the number of
required elements. Despite these advantages, applying DG methods to
transport problems in a sweeping framework remains challenging, as
there are currently no established results on achieving efficient
sweeping for DG methods on polytopal meshes. This limitation poses
difficulties for extending these DG approaches to sweeping-based,
parallelised solutions, particularly for large-scale, complex
geometries.

\subsection{Our Contribution}

This work addresses the challenge of enabling efficient sweeping
algorithms on unstructured polytopal meshes, providing the first
rigorous demonstration that a specific family of such meshes, those with a Voronoi structure, can
support systematic sweeping in any transport direction without
inducing cyclical dependencies in the underlying numerical
system. This result represents an advancement in the use of polytopal
meshes for transport problems, opening the door to leveraging DG
methods on polytopal elements in mesh-sweeping frameworks.

Our approach includes the development of a specialised scheduling
algorithm that effectively orders elements in a sweep-compatible
manner, ensuring that dependencies are resolved sequentially across
the mesh. This enables mesh-sweeping without the need for
cycle-breaking modifications in the solver, which are typically
required for unstructured or irregular meshes, enhancing computational
efficiency. This makes it feasible to extend sweeping-based,
parallelised solutions to complex and large-scale geometries that
require adaptive meshing and localised refinement.

Our results establish a framework that allows DG methods on polytopal Voronoi
meshes to be implemented efficiently in parallel settings. This
maximises computational performance on modern architectures and
addresses the scalability challenges that arise when fine
discretisations generate substantial numbers of unknowns. This work
represents a significant step toward practical and efficient BTE
simulations on complex geometries using DG methods on more general polytopal
elements.

\subsection{Related Literature}

The DOM, though powerful, was historically
limited by its sequential transport sweep, which was initially deemed
computationally prohibitive. The Koch-Baker-Alcouffe (KBA) method
\cite{baker1998sn} introduced parallelisation for transport sweeps on
high-performance computing architectures. Further enhancements, such
as the simultaneous multi-octant sweep
\cite{AdamsAdamsHawkinsSmithRauchwergerAmatoBaileyFalgout:2013} and
extensions to Krylov-based iterative solvers
\cite{doi:10.13182/NT171-171}, have advanced the efficiency of neutron
transport simulations. Techniques involving aggregation on
orthogonal-like meshes have demonstrated substantial performance gains
on modern architectures \cite{osti_1312619, osti_877776}. However,
these approaches typically rely on structured or orthogonal meshes,
which are less suited to the irregular geometries commonly encountered
in practical applications like neutronics.

Recent efforts have explored methods for mesh sweeping on irregular
and unstructured meshes, particularly simplicial meshes. For example,
sweeping feasibility has been demonstrated on regular 2D meshes
without hanging nodes \cite{camminady2018short} and topological
sorting with cycle-breaking has been applied to irregular 2D
triangulations \cite{vermaak2021}. Finite element solvers have also
been developed for irregular tetrahedral meshes with reduced storage
requirements, though these methods lack flexibility for adaptive
meshing frameworks such as {\it hp}-DG
\cite{CangianiGeorgoulisHouston:2014} or recovered finite elements
\cite{PolyRFEM}.

Polytopal meshes, due to their adaptability and efficient
representation of complex domains, have garnered attention as
alternatives to traditional structured meshes. Distorted hexahedral
meshes have been presented for cubic and spherical domains
\cite{GAO2022110964}, while semi-structured polyhedral meshes with
optimised parallel sweeps have been introduced to improve processor
efficiency \cite{Adams_2020}. Additional domain decomposition
strategies for enhanced processor efficiency have also been explored
\cite{COLOMER2013118, vermaak2021}.

The potential of DG methods on polygonal and polyhedral meshes has
also attracted recent interest due to their high-order accuracy,
adaptability for complex geometries and efficiency in reducing the
number of elements for irregular domains
\cite{AntoniettiCangianiCollisDongGeorgoulisGianiHouston:2016,ManziniRussoSukumar:2014}.
Recent work has further developed iterative solution techniques for
{\it hp}-DG schemes applied to the poly-energetic BTE, including
preconditioned Richardson and GMRES-based methods that yield mesh and
polynomial independent error estimates, enhancing their robustness and
efficiency \cite{HoustonHubbardRadley:2024}.

However, adapting DG methods to transport
problems with mesh-sweeping algorithms presents significant
challenges, particularly because sweeping requires a lower triangular
form that is often difficult to achieve on unstructured or polytopal
meshes. While methods such as the ``number-of-upstreams'' approach
permit cell evaluation when all upstream boundary conditions are
satisfied \cite{plimpton2005}, cycle-breaking strategies are often
necessary to manage dependencies on irregular meshes, as in
steady-state sweeps or time-lagged approaches for time-dependent
sweeps \cite{GAO2022110964, vermaak2021}. These modifications
introduce inefficiencies, limiting the application of sweeping in DG
settings on complex geometries.

Our work circumvents these limitations by establishing a class of
sweep-compatible polytopal meshes that require no cycle-breaking
adjustments, thus enabling the direct application of sweeping
algorithms in broader, more complex geometries. This contribution is
the first to rigorously demonstrate sweep-compatible polytopal meshes,
enhancing the applicability of polygonal DG and finite volume methods
in complex domains, and providing a robust framework for efficient and
cycle-free solutions to transport equations on irregular meshes.

The structure of this paper is as follows: In
Section~\ref{section:problem_setting}, we introduce notation, define
the underlying meshes and provide background on sweeping the DG method
in a simplified linear transport setup. Section~\ref{sec:BTE} presents
the DOM for the BTE and summarises
relevant existing results. In Section~\ref{sec:main_contributions}, we
introduce the necessary graph-theoretic concepts to show our main
result. Section~\ref{section:numerics} showcases numerical experiments
that assess the algorithm's robustness, computational complexity,
solution accuracy and applicability to a reactor core test case. We
conclude in Section \ref{section:summary}.

\section{Problem setup}
\label{section:problem_setting}

Let $\leb p(D)$, $1\le p\le \infty$ denote the standard Lebesgue
spaces for an open subset $D\subset\mathbb{R}^d$, $d=1,2,3$, with
corresponding norms $\|\cdot\|_{L^p(D)}$. We then introduce the
Sobolev spaces \cite{evans2022partial}
\begin{equation}
  \sobh{k}(D) 
  := 
  \ensemble{\phi\in\leb{2}(D)}
  {\D^{\vec\alpha}\phi\in\leb{2}(D), \text{ for } \norm{\geovec\alpha}\leq k},
\end{equation}
which are equipped with norms and semi-norms
\begin{gather}
  \Norm{u}_{\sobh{k}(D)}^2 
  := 
  \sum_{\norm{\vec \alpha}\leq k}\Norm{\D^{\vec \alpha} u}_{\leb{2}(D)}^2 
  \AND \norm{u}_{k}^2 
  :=
  \norm{u}_{\sobh{k}(D)}^2 
  =
  \sum_{\norm{\vec \alpha} = k}\Norm{\D^{\vec \alpha} u}_{\leb{2}(D)}^2
\end{gather}
respectively, where $\vec\alpha = \{ \alpha_1,...,\alpha_d\}$ is a
multi-index, $\norm{\vec\alpha} = \sum_{i=1}^d\alpha_i$ and
derivatives $\D^{\vec\alpha}$ are understood in a weak sense.

Let $\W \subset \mathbb{R}^d$, $d\ge2$ be a bounded, polytopal domain
with boundary $\Gamma$. Let $\vec{\nu} = \vec{\nu}(\vec{x})$ denote the outward pointing
normal of $\partial \Omega$ consisting of in-flow and out-flow
boundaries
\begin{equation}
  \label{in-out-flow}
  \begin{split}
    \Gamma_{-}
    &=
    \{\vec{x} \in \partial \Omega : \vec{\omega} \cdot \vec{\nu}(\vec{x}) < 0\}
    \quad (\text{in-flow}) \\
    \Gamma_{+}
    &=
    \{\vec{x} \in \partial \Omega : \vec{\omega} \cdot \vec{\nu}(\vec{x}) > 0\}
    \quad (\text{out-flow}).
  \end{split}
\end{equation}
Let $\vec \omega\in\mathbb R^d$ and $\Sigma_t \in\leb{\infty}(\W)$ be
given $f\in L^2(\Omega)$ and $g_D\in L^2(\Gamma_-)$. We consider
the boundary value problem to seek $u$ such that
\begin{equation}
  \label{transport:equation}
  \begin{split}
    \vec{\omega} \cdot \nabla u + \Sigma_t u &= f
    \text{ in } \Omega,
    \\
    u &= g_D \text{ on } \Gamma_{-}.
  \end{split}
\end{equation}
We make the standard assumption that there exists $\sigma_0$ such that
\begin{equation}
  \Sigma_t(\vec x)  \geq \sigma_0 > 0 \text{ for a.e. } \vec x\in\Omega,
\end{equation}
which is sufficient to guarantee the problem is well posed \cite{houston2001stabilised}.

\subsection{Polytopal Voronoi Meshes}
\label{subsection:polytopal_meshes}

We define a polytopal mesh (and Voronoi tessellation) over
$\Omega$ in a similar fashion to the definitions in
\cite{di2020hybrid, PolyRFEM}.

\begin{definition}[Polytopal Mesh]
  \label{definition:polytopal_mesh}  
  A {\it polytopal mesh} over $\Omega$ is characterised by a pair
  $(\mathcal{T}, \mathcal{E})$, such that:
  \begin{enumerate}
  \item The {\it element set}, $\mathcal{T}$, is a subdivision of
    $\Omega$ into disjoint $d$-dimensional polytopes such that:    
    \begin{equation}
      \overline{\Omega} = \bigcup_{T\in\mathcal{T}} \overline{T}.
    \end{equation}
  \item The {\it facet set}, $\mathcal{E}$, is the set of all
    $(d-1)$-dimensional facets associated with the subdivision,
    $\mathcal{T}$, including the boundary, $\partial\Omega$. We note
    that the elements of $\mathcal{E}$ form a partition of the mesh
    skeleton defined by $\bigcup_{T\in\mathcal{T}} \partial T$.
  \end{enumerate}
  Furthermore, for any mesh element $T \in \mathcal{T}$, let
  \begin{equation}
    \mathcal{E}_T := \{ F \in \mathcal{E} \mid F \subset \partial T \}
  \end{equation}
  denote the set of facets contained in $\partial T$. We further
  define $\mathcal{E}_{T_1, T_2} = \mathcal{E}_{T_1} \cap
  \mathcal{E}_{T_2}$ to be the set containing the facet(s) (if
  non-empty) between two mesh elements, $T_1$ and $T_2$. We slightly
  abuse notation by using $\mathcal{E}_{T_1, T_2}$ to refer to both
  the set containing the facet(s) and the facet itself if there is
  only one. This concept naturally extends to the boundary facets
  through
  \begin{equation}
    \mathcal{E}_{T, \partial\Omega} = \{ F \in \mathcal{E}_T \mid F \subset \partial \Omega \}.
  \end{equation}
  We define the neighbouring set as the set of elements that
  share a facet with a given element $T$:
  \begin{equation}    
    \mathcal{N}_T := \{ T' \in \mathcal{T} \mid \mathcal{E}_{T, T'} \neq \emptyset \}.    
  \end{equation}  
  We also define $\vec{\nu}_{T}(\vec{x})$ to be the outward
  pointing unit normal vector to $\partial T$ at $\vec{x} \in \partial
  T$ for $T \in \mathcal{T}$, and $\vec{\nu}_{T_1, T_2} :=
  \vec{\nu}_{1,2}$ to be the outward pointing unit normal vector to
  $\partial T_1$ defined over the element(s) in $\mathcal{E}_{T_1,
    T_2}$.
\end{definition}
    
\begin{definition}[Voronoi Tessellation]
  \label{definition:voronoi_tessellation} 
  A {\it Voronoi tessellation} over a $d$-dimensional polytopal
  domain $\Omega$ is a polytopal mesh $(\mathcal{T}, \mathcal{E})$
  where each element $T_i \in \mathcal{T}$ has an associated centre,
  $\vec{v}_i \in \mathbb{R}^d$, such that $T_i$ consists of the points
  which are strictly closest to its corresponding centre $\vec{v}_i$
  or simply written:
  \begin{equation}
    T_i = \{\vec{x} \in \Omega : \|\vec{x} - \vec{v}_i\|_2 < \|\vec{x} - \vec{v}_j\|_2 \text{ for all } j \ne i\}.
  \end{equation}
\end{definition}

\begin{remark}[Convexity of Voronoi elements]
  We note that the convexity of the elements of a Voronoi tessellation
  implies that the facet set between two elements has at most one
  element.
\end{remark}

\subsection{Discontinuous Galerkin method for linear transport}
\label{subsection:fvms_transport}

To motivate the use of mesh sweeping algorithms, we introduce a DG
approximation of \eqref{transport:equation}
\cite{CangianiGeorgoulisHouston:2014}. Let $(\mathcal{T},
\mathcal{E})$ be a polytopal mesh.

For $p \in \naturals_0$ and an element $T\in\mathcal T$ we denote the
set of polynomials of total degree at most $p$ by $\poly{p}(T)$ which
allows us to define the discontinuous Galerkin finite element space
\begin{equation}
  \fes_p := \ensemble{v_h \in \leb{2}(\W)}{v_h\vert_{T} \in \poly{p}(T) \Foreach T\in\mathcal T}.
\end{equation}
Let $T_1, T_2 \in \mathcal T$ be two elements sharing a facet
$\mathcal E_{T_1, T_2}$, where $T_1$ is upwind of $T_2$. For a function $v:\Omega\to\mathbb R$ we define the jump
and upwind jump operators
\begin{equation}
    \jump{v}
    =
    v\vert_{T_1} \vec \nu_{T_1}
    +
    v\vert_{T_2} \vec \nu_{T_2},
    \qquad
    \ujump{v}
    =
    v\vert_{T_1}
    -
    v\vert_{T_2}      
\end{equation}
respectively, noting that they can differ by at most a sign. Let $h_T
:= \diam{T}$ and $h \in \fes_0$ to be the piecewise constant mesh size
function such that $h|_T = h_T$.

For a $T \in \mathcal T$ we define
\begin{equation}
  \begin{split}
    \partial_- T
    &=
    \{\vec{x} \in \partial T : \vec{\omega} \cdot \vec{\nu}(\vec{x}) < 0\}
    \\
    \partial_+ T
    &=
    \{\vec{x} \in \partial T : \vec{\omega} \cdot \vec{\nu}(\vec{x}) > 0\}
    \\
    \partial_0 T
    &=
    \{\vec{x} \in \partial T : \vec{\omega} \cdot \vec{\nu}(\vec{x}) = 0\}
\end{split}
\end{equation}
as elementwise inflow, outflow and characteristic components of the
element. For a non-degenerate polygon, at least one face is an inflow and one is an outflow.

Now we define the discontinuous Galerkin method to find $u_h \in
\fes_p$ such that
\begin{equation}
  \label{eq:dgmethod}
  \bi{u_h}{v_h}{\vec \w} = l_{\vec \w}(v_h) \Foreach v_h \in \fes_p
\end{equation}
where
\begin{equation}
  \bi{u_h}{v_h}{\vec w}
  :=
  \sum_{T\in\mathcal T}
  \qp{\int_T \qp{\vec \w \cdot \nabla u_h + \Sigma_t u_h} v_h \d \vec x
    -
    \int_{\partial_- T \slash \Gamma}
    \vec \w \cdot \vec \nu \ujump{u_h} v_h \d s
    -
    \int_{\partial_- T \cap \Gamma_-} \vec \w \cdot \vec \nu u_h v_h
  }
\end{equation}
and
\begin{equation}
  l_{\vec \w}(v_h)
  :=
  \int_\W
  f v_h
  -
  \int_{\Gamma_-} \vec \w \cdot \vec \nu g_D v_h.
\end{equation}
A natural notion of error is the energy norm defined by
\begin{equation}
    \label{eqn:dg_energy_norm}
  \enorm{u_h}^2
  :=
  \sum_{T\in \mathcal T}
  \qp{
  \Norm{\sigma_0^{1/2} u_h}_{\leb{2}(T)}^2
  +
  \frac 12
  \Norm{{\norm{\vec w \cdot \vec \nu}}^{1/2} \ujump{u_h}}_{\leb{2}(\partial_- T \slash \Gamma)}^2
  +
  \frac 12
  \Norm{{\norm{\vec w \cdot \vec \nu}}^{1/2} {u_h}}_{\leb{2}(\partial_- T \cap \Gamma_-)}^2
  }.
\end{equation}

\begin{lemma}[Error control]
  If $u\in\sobh r(\W)$, $1\le r\in\mathbb{N}$, then a priori error
  control in this norm is given by
  \begin{equation}
    \enorm{u - u_h}
    \leq
    C \max_T h_T^{\min\qp{p+1,r} - \tfrac 12} \norm{u}_{\sobh{r}(\W)}
  \end{equation}
  and, in general, no improvement is anticipated in the weaker
  $\leb{2}$-norm
  \cite{bey1996hp,johnson1986analysis,lesaint1974finite}. One can also
  show results in a more general class of solutions
  \cite{suli1997posteriori}, however we do not further investigate
  this here.
\end{lemma}

\subsection{Mesh Sweeping for Transport Problems}
\label{subsection:mesh_sweeping}

The DG method \eqref{eq:dgmethod} yields a system of equations that
can be expressed in an algebraic form. Specifically, it leads to a
linear system of the form
\begin{equation}
  \vec{A} \vec{u} = \vec{f},
\end{equation}
where $\vec{A}$ is the global system matrix, $\vec{u}$ is the
vector of unknowns corresponding to the DG coefficients of the
approximate solution, and $\vec{f}$ represents the source terms and
boundary conditions.

For each element $T \in \mathcal{T}$, the DG formulation
\eqref{eq:dgmethod} includes terms that depend on neighbouring elements
that are upwind in the transport direction $\vec{\omega}$. This
dependency structure is reflected in the global system matrix
$\vec{A}$, where entries corresponding to interactions between
neighbouring elements are determined by the orientation of
$\vec{\omega}$ relative to the outward normal vectors on element
interfaces.

To solve this system efficiently, we seek to reorder the mesh elements
to achieve a block lower triangular form. This reordering is
accomplished by applying a permutation matrix $\vec{P}$ that aligns
the elements according to the causality imposed by the transport
direction. The resulting reordered system,
\begin{equation}
  \vec{P} \vec{A} \vec{u} = \vec{P} \vec{f},
\end{equation}
yields a block lower triangular structure for \(\vec{P} \vec{A}\):
\begin{equation}
  \vec{P} \vec{A} = 
  \begin{bmatrix}
    \vec{A}_{11} & \vec{0} & \cdots & \vec{0} \\
    \vec{A}_{21} & \vec{A}_{22} & \cdots & \vec{0} \\
    \vdots & \vdots & \ddots & \vdots \\
    \vec{A}_{m1} & \vec{A}_{m2} & \cdots & \vec{A}_{mm}
  \end{bmatrix}.
\end{equation}
This lower triangular form ensures that the solution on each element
depends only on elements that are ``upstream'' in the propagation
direction. Consequently, we can use a forward substitution approach,
often called a ``sweep'', to compute the solution sequentially, with
each block solved independently once its dependencies have been
resolved.

Achieving a block lower triangular structure through reordering is
typically feasible only for structured or semi-structured meshes,
where element adjacency follows a predictable pattern. For general
unstructured meshes, cyclic dependencies in the mesh may prevent a
complete ordering that respects the transport direction. In such
cases, no reordering can yield a strictly lower triangular form
without additional modifications. These cycles may necessitate
alternative strategies, such as iterative solvers or cycle-breaking
techniques, which approximate the ideal dependency ordering without
achieving a fully block lower triangular matrix. We will explore
methods for handling these challenges in Section
\ref{sec:main_contributions}.

When a block lower triangular structure is achievable, the linear
system can be solved in a single forward pass, minimising the need for
iterative solution methods and reducing computational costs,
particularly in high-dimensional problems or on fine
discretisations. In cases where exact sweeping is infeasible,
approximate sweeping or iterative solvers with directional
preconditioning may provide an effective alternative, leveraging the
transport direction for accelerated convergence.

\section{Boltzmann Transport Equation}
\label{sec:BTE}

In the previous section, we developed a DG method for solving the
transport equation for a single direction, $\vec{\omega} \in
\mathbb{S}^{d-1}$. The BTE represents transport solutions across
multiple directions simultaneously, where both the spatial variable
$\vec{x}$ and the angular variable $\vec{\omega}$ are independent. In
this section, we reuse much of the ideas and notation, although extended
into a higher dimensional space.

To that end, let $\Sigma_s$ denote the differential scattering kernel
and $\Sigma_t$ the macroscopic total cross-section, both of which
depend on the spatial variable $\vec{x}$. We assume that there exists
a constant $\sigma_0 > 0$ such that
\begin{equation}
  \label{eq:coerc-assumption}
  \Sigma_t(\vec{x}) - \int_{\mathbb{S}^{d-1}} \Sigma_s(\vec{x}, \vec{\omega} \cdot \vec{\omega}') \, d\vec{\omega}' \geq \sigma_0 > 0 \quad \text{for a.e. } \vec{x} \in \Omega, \ \vec{\omega} \in \mathbb{S}^{d-1}.
\end{equation}
This coercivity assumption ensures the problem is well-posed by
providing a bound on the scattering that allows the existence of a
unique solution.

We focus on the mono-energetic form of the BTE, which is given by the
following integro-differential equations
\begin{equation}
  \label{eqn:bte}
  \begin{split}
    \vec{\omega} \cdot \nabla \psi(\vec{x}, \vec{\omega})
    +
    \Sigma_t(\vec{x}) \psi(\vec{x}, \vec{\omega})
    &=
    \frac{1}{|\mathbb{S}^{d-1}|} \int_{\mathbb{S}^{d-1}} \Sigma_s(\vec{x}, \vec{\omega} \cdot \vec{\omega}') \psi(\vec{x}, \vec{\omega}') \, d\vec{\omega}'
    +
    f(\vec{x}, \vec{\omega}),
    \\
    \psi(\vec{x}, \vec{\omega})
    &=
    g_D (\vec{x}, \vec{\omega}) \quad \text{on } \Gamma_{-},
  \end{split}
\end{equation}
where $|\mathbb{S}^{d-1}|$ denotes the surface area of the
$(d-1)$-sphere. The inflow boundary condition extends across
directions as
\begin{equation}
  \label{eqn:bte_bcs}
  \Gamma_{-} = \{ (\vec{x}, \vec{\omega}) \in \Omega \times \mathbb{S}^{d-1} : \vec{\omega} \cdot \vec{n}(\vec{x}) < 0 \}.
\end{equation}

\subsection{Discrete Ordinates Method}

To approximate the BTE over discrete directions, we apply a discrete
ordinates method. For a given $N_Q > 0$, let
$\{\vec{\omega}_k\}_{k=1}^{N_Q}$ denote a set of discrete directions
with corresponding quadrature weights $\{w_k\}_{k=1}^{N_Q}$, such that
for a generic function $z$,
\begin{equation}
  \frac{1}{|\mathbb{S}^{d-1}|} \int_{\mathbb{S}^{d-1}} z(\vec{\omega}) \, d\vec{\omega} \approx \sum_{k=1}^{N_Q} w_k z(\vec{\omega}_k).
\end{equation}
This discretisation allows us to reduce \eqref{eqn:bte} to a coupled
system of transport equations, each defined over a specific direction
$\vec{\omega}_k$.

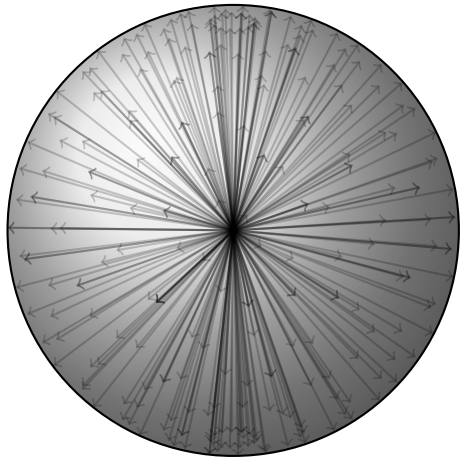
\begin{figure}[h!]
  \centering
  \tdplotsetmaincoords{70}{110}
  \begin{tikzpicture}[tdplot_main_coords, scale=3]
  \shade[ball color = gray!20] (0,0,0) circle (1cm);
  \draw[thick] (0,0,0) circle (1cm);

  \pgfmathsetmacro{\numrays}{15} 
  \pgfmathsetmacro{\rayangleincrement}{360/\numrays}
  \pgfmathsetmacro{\numcircles}{15} 
  \pgfmathsetmacro{\circleangleincrement}{360/\numcircles}

  \def\rayopacity{0.2}

  \foreach \angle in {0, \rayangleincrement, ..., 360} {
    \foreach \theta in {0, \circleangleincrement, ..., 360} {
      \tdplotsetrotatedcoords{\angle}{\theta}{0}
      \draw[thick, tdplot_rotated_coords, ->, opacity=\rayopacity] (0,0,0) -- (1,0,0); 
    }
  }
    
\end{tikzpicture}
  \caption{An illustration of discrete ordinates discretisation of the unit sphere, highlighting the angular partitioning used for transport directions.}
\end{figure}

We denote the angular approximation $\psi_k(\vec{x}) \approx
\psi(\vec{x}, \vec{\omega}_k)$ for each direction $\vec{\omega}_k$,
leading to the following system of equations. For each $k = 1, \ldots,
N_Q$, we seek $\psi_k$ such that
\begin{equation}
  \label{eqn:bte_dom_disc}
  \vec{\omega}_k \cdot \nabla \psi_k + \Sigma_t \psi_k = \sum_{l=1}^{N_Q} w_l \Sigma_s(\vec{x}, \vec{\omega}_k \cdot \vec{\omega}_l) \psi_l + f_k,
\end{equation}
where $f_k$ represents the source term in the direction
$\vec{\omega}_k$. Each equation in this system resembles the transport
equation solved in the previous section, making it suitable for the DG
formulation described above.

To apply the DG scheme, we introduce a scattering bilinear form $s_{h,
  \vec{\omega}_k}(\psi_h, v_h)$ to capture the coupling between
directions. We define the concatenated vector of directional fluences
as
\begin{equation}
  \Psi := \begin{bmatrix} \psi_{1} \\ \psi_{2} \\ \vdots \\ \psi_{N_Q} \end{bmatrix},
  \quad
  \Psi_h := \begin{bmatrix} \psi_{h,1} \\ \psi_{h,2} \\ \vdots \\ \psi_{h,N_Q} \end{bmatrix},
\end{equation}
and let
\begin{equation}
  s_{h, \vec{\omega}_k}(\Psi_h, v_h)
  :=
  \int_\Omega \sum_{l=1}^{N_Q} w_l \Sigma_s(\vec{\omega}_k \cdot \vec{\omega}_l) \psi_{h,l} v_h \, d\vec{x}.
\end{equation}
Then, for each direction $k = 1, \dots, N_Q$, we seek $\psi_{h,k} \in
\fes_p$ such that
\begin{equation}
  \label{eqn:bte-sys}
  \bi{\psi_{h,k}}{v_{h,k}}{\vec{\omega}_k}
  =
  s_{h, \vec{\omega}_k}(\Psi_h, v_{h,k})
  +
  l_{\vec{\omega}_k}(v_{h,k}) \quad \Foreach v_{h,k} \in \fes_p.
\end{equation}
An important consequence of this formulation is the ability to
quantify the error in the finite element approximation. The following
lemma establishes an a priori error bound in the energy norm for the
DG approximation.
\begin{lemma}[Error control]
    \label{lemma:error_control}
    Let $\psi_k \in \sobh{r}(\Omega)$ for $r \in \mathbb{N}$ denote
    the solution of the semi-discrete problem \eqref{eqn:bte_dom_disc}
    for each $k=1, \dots, N_Q$, and let $\psi_{h,k} \in \fes_p$ be
    the finite element approximation from \eqref{eqn:bte-sys}. Then,
    \begin{equation}
      \enorm{\psi_k - \psi_{h,k}}
      \leq
      C h^{\min(p+1, r) - \tfrac{1}{2}} \norm{\psi_k}_{\sobh{r}(\Omega)}.
    \end{equation}
    This result provides a bound on the DG approximation error, which
    can be further used to analyse the fully discrete error
    \cite{Wang_Sheng_Han_2016}.
\end{lemma}

\subsection{Mesh Sweeping for BTE}
\label{subsection:mesh_sweeping_BTE}

For the discrete ordinates BTE, the DG formulation \eqref{eqn:bte-sys}
yields a coupled block system of the form
\begin{equation}
  \label{eqn:full_matrix_system}
  \vec{A} \vec{\Psi}_h = \vec{S} \vec{\Psi}_h + \vec{f},
\end{equation}
where $\vec{A}$ is block diagonal and corresponds to the discrete
transport operators, and $\vec{S}$ has a dense $N_Q \times N_Q$ block
structure representing scattering interactions that inherently couple
different directions.

Due to the block diagonal structure of $\vec{A}$, a direct solution of
\eqref{eqn:full_matrix_system} may be computationally
expensive. Instead, we employ a fixed-point method known as
\emph{source iteration}. Given an initial guess $\vec{\psi}^{(0)}$,
source iteration proceeds by solving
\begin{equation}
  \vec{A} \vec{\Psi}_h^{(n)} = \vec{S} \vec{\Psi}_h^{(n-1)} + \vec{f}
\end{equation}
at each iteration. It should be noted that there are other approaches
to accelerate the convergence of this fixed-point method
\cite{HoustonHubbardRadley:2024} as well as diffusion synthetic
acceleration
\cite{southworth2020diffusionsyntheticaccelerationheterogeneous}. We
will focus on the source iteration in this work.

To facilitate an efficient solution, we apply the permutation matrix
$\vec{P}$ introduced in the previous section, which reorders the
system to align with transport dependencies. This reordering yields a
modified system
\begin{equation}
    \label{eqn:swept_full_system}
    \vec{P} \vec{A} \vec{\Psi}_h^{(n)} = \vec{P} \vec{S} \vec{\Psi}_h^{(n-1)} + \vec{P} \vec{f},
\end{equation}
where $\vec{P} \vec{A}$ is now lower triangular rather than block
diagonal. This structure allows each directional equation to be solved
sequentially within the sweep, using an efficient forward substitution
approach instead of a full matrix inversion.

However, as was the case in mono-directional transport, the ability to
construct a permutation matrix $\vec{P}$ that achieves this lower
triangular form depends on the mesh structure. Not all meshes allow a
reordering that aligns the transport dependencies without introducing
cyclic dependencies in the system. Consequently, a key objective of
this work is to identify a class of meshes for which such a reordering
is always possible, enabling the efficient use of mesh sweeping
techniques on complex geometries. This is the goal of the next
section.

\section{Cycle-Free Meshes and Voronoi Tessellations}
\label{sec:main_contributions}

In this section, we demonstrate that certain existing classes of
polytopal meshes possess properties that enable cycle-free
sweeps. This approach contrasts with previous methods that require
either specific spatial decompositions~\cite{baker1998sn} or
cycle-breaking techniques~\cite{vermaak2021}.

For a first-order hyperbolic PDE with convection in direction
$\vec{\omega}$, flux can only flow from an upwind to a downwind region
with respect to $\vec{\omega}$ over the domain. In the spatially
discrete case, this requirement translates to flow traversing each
facet of an element in our Voronoi tessellation in a single
direction. This induces a natural dependency structure among the
elements based on the fluxes. To characterise these dependencies, we
introduce the concept of a directed dual mesh.

\begin{figure}[h!]
  \begin{center}
    \tikzset{every picture/.style={line width=0.75pt}} 

\begin{tikzpicture}[x=0.75pt,y=0.75pt,yscale=-0.5,xscale=0.5]

\draw  [line width=1.5]  (248.67,129.51) -- (267.07,244.29) -- (146.46,330.88) -- (60.61,256.37) -- (81.05,145.62) -- (164.43,114.9) -- cycle ;
\draw  [line width=1.5]  (267.07,244.29) -- (334.11,290.81) -- (344.74,393.3) -- (231.91,451.9) -- (170.58,411.63) -- (146.46,330.88) -- cycle ;
\draw [color={rgb, 255:red, 126; green, 211; blue, 33 }  ,draw opacity=1 ][line width=1.5]    (146.46,330.88) -- (267.07,244.29) ;
\draw [color={rgb, 255:red, 208; green, 2; blue, 27 }  ,draw opacity=1 ][line width=1.5]    (206.76,287.59) -- (262.1,356.54) ;
\draw [shift={(264.6,359.66)}, rotate = 231.25] [fill={rgb, 255:red, 208; green, 2; blue, 27 }  ,fill opacity=1 ][line width=0.08]  [draw opacity=0] (24.96,-6.24) -- (0,0) -- (24.96,6.24) -- cycle    ;
\draw [line width=1.5]    (34.2,357.83) -- (122.19,385.93) ;
\draw [shift={(126,387.15)}, rotate = 197.71] [fill={rgb, 255:red, 0; green, 0; blue, 0 }  ][line width=0.08]  [draw opacity=0] (24.96,-6.24) -- (0,0) -- (24.96,6.24) -- cycle    ;
\draw [color={rgb, 255:red, 208; green, 2; blue, 27 }  ,draw opacity=1 ][line width=1.5]    (206.76,287.59) -- (146.86,211.2) ;
\draw [shift={(144.39,208.05)}, rotate = 51.89] [fill={rgb, 255:red, 208; green, 2; blue, 27 }  ,fill opacity=1 ][line width=0.08]  [draw opacity=0] (24.96,-6.24) -- (0,0) -- (24.96,6.24) -- cycle    ;
\draw  [color={rgb, 255:red, 208; green, 2; blue, 27 }  ,draw opacity=1 ][fill={rgb, 255:red, 208; green, 2; blue, 27 }  ,fill opacity=1 ] (203.08,287.59) .. controls (203.08,285.54) and (204.73,283.89) .. (206.76,283.89) .. controls (208.8,283.89) and (210.45,285.54) .. (210.45,287.59) .. controls (210.45,289.63) and (208.8,291.28) .. (206.76,291.28) .. controls (204.73,291.28) and (203.08,289.63) .. (203.08,287.59) -- cycle ;
\draw [line width=1.5]  [dash pattern={on 5.63pt off 4.5pt}]  (384,30) -- (385.2,460.3) ;
\draw [line width=1.5]    (486,176) -- (569.8,288.1) ;
\draw [shift={(572.2,291.3)}, rotate = 233.22] [fill={rgb, 255:red, 0; green, 0; blue, 0 }  ][line width=0.08]  [draw opacity=0] (24.96,-6.24) -- (0,0) -- (24.96,6.24) -- cycle    ;

\draw (283.3,140.63) node  [font=\LARGE] [align=left] {$\displaystyle T_{1}$};
\draw (322.26,440.63) node  [font=\LARGE] [align=left] {$\displaystyle T_{2}$};
\draw (285.06,385.69) node  [font=\LARGE,color={rgb, 255:red, 208; green, 2; blue, 27 }  ,opacity=1 ] [align=left] {$\displaystyle \ensuremath{\boldsymbol{\nu}}_{1,2}$};
\draw (83.2,350.8) node  [font=\LARGE] [align=left] {$\displaystyle \ensuremath{\boldsymbol{\omega}} $};
\draw (175.78,192.25) node  [font=\LARGE,color={rgb, 255:red, 208; green, 2; blue, 27 }  ,opacity=1 ] [align=left] {$\displaystyle \ensuremath{\boldsymbol{\nu}} _{2,1}$};
\draw (149,50) node [anchor=north west][inner sep=0.75pt]  [font=\LARGE] [align=left] {Primal};
\draw (511,50) node [anchor=north west][inner sep=0.75pt]  [font=\LARGE] [align=left] {Dual};
\draw (475,145) node  [font=\LARGE] [align=left] {$\displaystyle T_1$};
\draw (591,315) node  [font=\LARGE] [align=left] {$\displaystyle T_2$};

\end{tikzpicture}
  \end{center}
  \caption{An illustration of upwind and downwind elements for $\vec{\omega} \cdot \vec{\nu}_{1,2} > 0$, with $T_1$ upwind of $T_2$, as shown by the edge direction on the dual graph.}
  \label{fig:upwind_meaning}  
\end{figure}
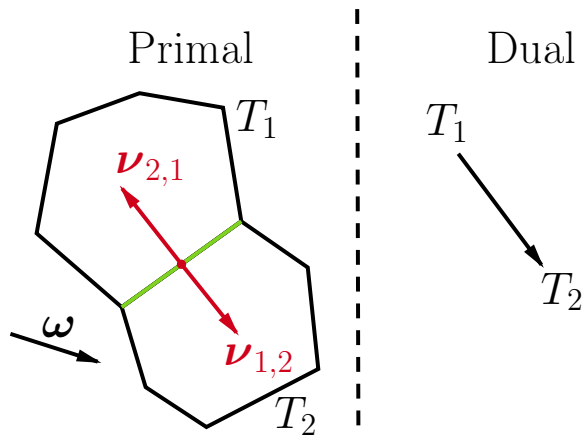

\begin{definition}[Dual mesh]
  \label{definition:dual_mesh}
  A {\it dual mesh} $\mathcal{M}^* = (\mathcal{T}^*, \mathcal{E}^*)$
  of a polytopal mesh $\mathcal{M} = (\mathcal{T}, \mathcal{E})$ is an
  undirected graph whose vertices $\mathcal{T}^*=\mathcal{T}$
  represent the elements of the primal mesh, and whose edges
  $\mathcal{E}^*$ represent shared boundaries between elements of the
  primal mesh, i.e., $\mathcal{E}^* = \{\{T_i,T_j\} \subseteq
  \mathcal{T} \colon \mathcal{E}_{T_i, T_j} \ne \emptyset\}$.
\end{definition}

If there is a known direction of flow $\vec{\omega}$ across the facets
of the primal mesh $\mathcal{E}$, this may be indicated through
directed edges, yielding a {\it directed dual mesh}. The direction of
flow between neighbouring elements $T_i$ and $T_j\in\mathcal{T}$ can be
determined by considering the sign of the dot product between
$\vec{\omega}$ and the outward-pointing normal vector of $T_i$ on
$\mathcal{E}_{i,j}$, denoted $\vec{\nu}_{i,j}$. Specifically, if
$\vec{\omega}\cdot\vec{\nu}_{i,j}(\vec{x}) > 0$ for
$x\in\mathcal{E}_{T_i, T_j}$, then $T_i$ is upwind of $T_j$, and the
directed edge points from $T_i$ to $T_j$.

\begin{definition}[Directed dual mesh]
  \label{definition:directed_dual_mesh}
  A {\it directed dual mesh} $\mathcal{M}^*_{\vec{\omega}} =
  (\mathcal{T}^*, \mathcal{E}^*)$ of a polytopal mesh $\mathcal{M} =
  (\mathcal{T}, \mathcal{E})$ with respect to a direction
  $\vec{\omega}$ is the directed graph with vertex set
  $\mathcal{T}^*=\mathcal{T}$ and edge set $\mathcal{E}^* =
  \{(T_i,T_j) \in \mathcal{T} \times \mathcal{T} \colon
  \mathcal{E}_{T_i, T_j} \ne \emptyset, \,
  \vec{\omega}\cdot\vec{\nu}_{i, j}(\vec{x}) > 0 \}$.  See
  Figure~\ref{fig:upwind_meaning} for an example of a mesh and its
  directed dual.
\end{definition}

\begin{remark}[Dual of a Voronoi mesh]
    \label{remark:voronoi_dual}
    Voronoi tessellations naturally admit a dual, the Delaunay
    triangulation, provided the tessellation is non-degenerate
    \cite{alma991001396089702761}. Additionally, the direction of flow
    along the dual edges is uniquely determined: each facet has a
    single direction of flow as described in
    Definition~\ref{definition:dual_mesh}. This is illustrated for
    $d=2$ in Figures~\ref{fig:upwind_meaning} and
    \ref{fig:primal_dual_mesh}.
\end{remark}

The directed dual mesh allows us to visualise the potential cycles of
flux in the mesh. To generate an appropriate permutation for the
linear system \eqref{eqn:swept_full_system}, we require a topological
sort of the dual mesh to order elements based on their upwind
dependencies. It can be shown that a topological sort exists if and
only if the directed graph is acyclic, i.e. is a DAG
\cite{bangJensen2009}. This property is essential: an acyclic dual
mesh results in a linear system that can be reordered to be lower
triangular.

In more complicated cases, i.e., the BTE, we require
$\mathcal{M}^*_{\vec{\omega}}$ to be acyclic for all directions
$\vec{\omega} \in \mathbb{S}^{d-1}$. If a mesh possesses this
property, we refer to it as {\it omnidirectionally acyclic}.

The following result shows that the Voronoi tessellations satisfy this
property through a reformulation of the upstream dependency condition
in Definition~\ref{definition:dual_mesh}.

\begin{figure}[h!]
  \begin{center}
    \input{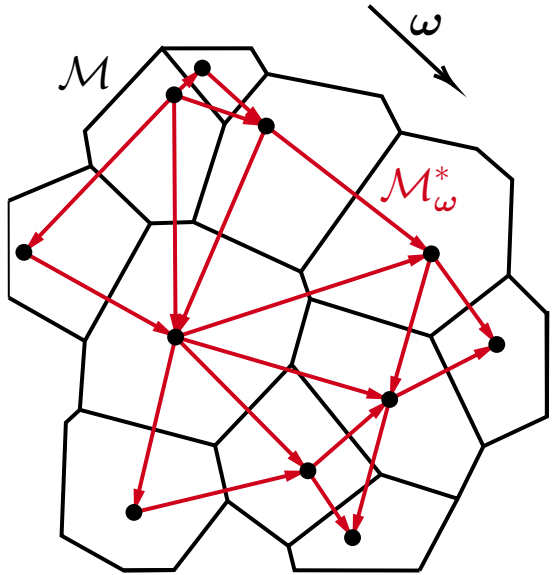}
  \end{center}
  \caption{Primal mesh elements, $\mathcal{T}$ (in black), with the associated directed dual mesh $(\mathcal{T}^, \mathcal{E}^)$, represented by black points and red arrows for a given direction $\vec{\omega}$.}
  \label{fig:primal_dual_mesh}
\end{figure}

\begin{theorem}[Voronoi meshes are omnidirectionally acyclic]
  \label{theorem:big_result}
  Voronoi meshes are omnidirectionally acyclic, i.e., there is no
  direction in which a cycle exists in the corresponding directed dual
  mesh.
\end{theorem}

\begin{proof}
  Let $\{ \vec{v}_i\}_i$ denote the Voronoi centres for the Voronoi mesh 
  $(\mathcal{T}, \mathcal{E})$. Since the vector between the centres of 
  neighbouring cells is perpendicular to the shared facet, we can express 
  the outward-pointing normal vector from $T_i$ to $T_j$ as
  \begin{equation}
    \label{equation:geom_normal}
    \vec{\nu}_{i,j} \coloneqq \vec{v}_j - \vec{v}_i.
  \end{equation}  

  Using this definition of normal vectors, we can rewrite
  the upstream condition from Definition~\ref{definition:dual_mesh} by
  defining $\delta_i = \vec{\omega} \cdot \vec{v}_i$ and noting
  \begin{equation}
    \label{deciding_point}
    \vec{\omega}\cdot \vec{\nu}_{ij} >0 \Leftrightarrow  \delta_i <  \delta_j.
  \end{equation}  

  Now assume by contradiction that the directed dual mesh has a cycle
  $T_1 \to T_2 \to \dots \to T_m \to T_1$.  If $\delta_i = \delta_j$ for $i \ne
  j$, the shared boundary between the corresponding elements $T_i$ and
  $T_j$ is parallel to the direction of flow and hence does not induce
  an upstream or downstream dependency. Otherwise, we obtain $\delta_1 <
  \delta_2 < \dots < \delta_m < \delta_1$, which is a contradiction.
\end{proof}

Algorithm~\ref{alg:Scheduler} computes a topological sort of the
directed dual of a Voronoi tessellation given its centres and has
several appealing properties.
\begin{algorithm}
  \caption{Voronoi-Scheduler}
  \label{alg:Scheduler}
  \begin{algorithmic}[1]                
    \State \textbf{Input:} Voronoi centers $\{\vec{v}_i\}_{i=1}^N$; direction $\vec{\omega}\in\mathbb{S}^{d-1}$.
    \State \texttt{distances} = []           
    \For{$i=1,\dots,n$}  
    \State \texttt{$\delta_i \gets \vec{\omega}\cdot\vec{v}_i$} 
    \State \texttt{distances.append($\delta_i$)}
    \EndFor
    \State \texttt{order $\gets$ argsort(distances)} 
    \State\Return \texttt{order}
  \end{algorithmic}
\end{algorithm}
First, it has a time complexity of $\mathcal{O}(N\ln N + dN)$ and a
space complexity of $\mathcal{O}(dN)$, where $d$ is the dimension of
the Voronoi centres and $N$ is the number of Voronoi
cells. Additionally, it can be executed in advance of a numerical
solver as an ``offline'' procedure. Most importantly,
Algorithm~\ref{alg:Scheduler} eliminates the need to consider cycles
in the subsequent numerical method, greatly simplifying the
implementation.

\begin{remark}[A Voronoi tessellation is sufficient but not necessary]
  We note that $\mathcal T$ being a Voronoi tessellation is a
  sufficient but not necessary condition for
  Algorithm~\eqref{alg:Scheduler} to compute a valid topological
  sort. Specifically, we require element-wise convexity and that
  Equation~\eqref{equation:geom_normal} holds, which are necessary
  conditions for a Voronoi tessellation. It is possible to extend this
  definition to include additional meshes, such as Pitteway
  triangulations (via their circumcentres), that can also be sorted
  using Algorithm~\eqref{alg:Scheduler}. We focus on Voronoi
  tessellations here to simplify our exposition.
\end{remark}

\begin{remark}[Parallelisation strategies]
  Definition~\ref{definition:voronoi_tessellation} allows us to take a
  subset $\Tilde{\Omega}\subset\Omega$ of a given Voronoi tessellation
  and apply Algorithm~\eqref{alg:Scheduler} to this subset (see
  Figure~\ref{fig:primal_dual_mesh_submesh}). This enables
  parallelisation via multi-meshing techniques, such as superimposing
  a KBA-style mesh over the original mesh and performing both the KBA
  sweep and the sub-sweep within each KBA element using
  Algorithm~\ref{alg:Scheduler}. We leave this direction for future
  research.
\end{remark}

\begin{figure}[h!]
  \begin{center}
    \input{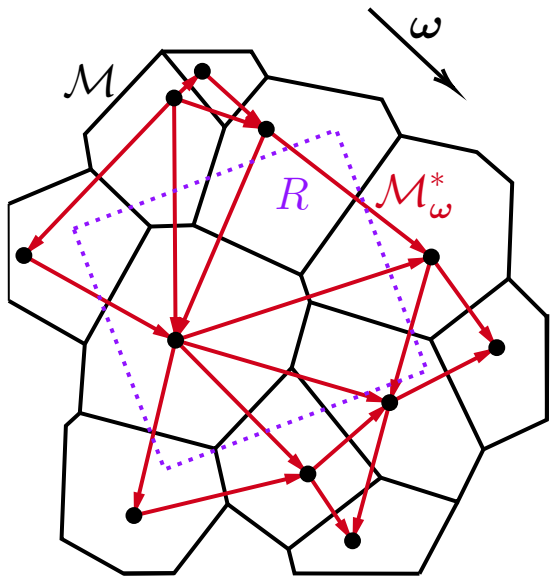}
  \end{center}
  \caption{An example of a subdomain problem: Algorithm~\ref{alg:Scheduler} sorts a subset of regions, $R$, from a full Voronoi decomposition, even when Voronoi centres lie outside $R$.}
  \label{fig:primal_dual_mesh_submesh}
\end{figure}

\section{Numerical Experimentation}
\label{section:numerics}

We present our method under various scenarios to test its robustness
and demonstrate its applicability. In \cite{HoustonHubbardRadley:2024}
a test case is provided that is challenging to approximate and
requires high resolution in both the spatial and angular components
due to steep solution gradients. We consider \eqref{eqn:bte} with the
coefficients $\Sigma_t(\vec{x}) = 1.0$ and $\Sigma_s(\vec{x}) = 0.7$
(unless otherwise stated). The source term is selected to yield the
solution
\begin{equation}
    \label{eqn:numerics_solution}
    \psi(\vec{x}, \vec{\omega}) = \exp\{-(\vec{x}\cdot\vec{\omega})^2\}.
\end{equation}
We evaluate the solution over the unit square, $\Omega = [0,1]^2$,
unless otherwise stated. In this section, we aim to answer several
questions related to the robustness and performance of our
approach. We assess Algorithm~\ref{alg:Scheduler}'s robustness under
varying numbers of discrete ordinates and spatial elements, and
examine the structure of the operator $\vec{P}\vec{A}$ from
Equation~\eqref{eqn:swept_full_system}. We also demonstrate the
convergence of the fixed-point iterative scheme and the numerical
method under $h$- and $k$-refinement. Finally, we showcase the
method's applicability to a complex geometry by solving for the scalar
flux solution over a representative domain of a nuclear reactor core.

\subsection{Robustness with Respect to Discrete Ordinates}

We first examine the algorithm's scalability when applied to the
discrete ordinates problem presented in
\eqref{eqn:full_matrix_system}. We report CPU times for the complete
topological sort as a function of the number of spatial elements, $N$,
for various numbers of discrete ordinates, $N_Q$. We anticipate
log-linear scaling in $N$ and linear scaling in $N_Q$, leading to an
overall complexity of $\mathcal{O}(N_Q N \ln N)$.

As observed in \cite{PearceKelly:2007}, optimisations may be achieved
by reusing previous sorts. For instance, if a sweep in direction
$\vec{\omega}_k$ has been completed, a sweep in direction
$-\vec{\omega}_k$ is trivial due to the property of
\eqref{deciding_point}. These optimisations were not implemented in
this study, as our aim is to demonstrate robustness across $N_Q$ and
$N$.

\begin{figure}[h!]
    \centering
    \begin{tikzpicture}
    \begin{axis}[
        thick,
        xmode=log,
        ymode=log,
        xlabel= {Number of Elements, $N$},
        ylabel= {CPU Time [s]},
        grid=both,
        minor grid style={gray!25},
        major grid style={gray!25},
        legend style={draw=none, at={(1, 0)}, anchor=south east, font=\small},
        legend pos=outer north east,
      ]
      \addlegendimage{empty legend}
      \addlegendentry{\hspace{-.6cm}\textbf{No. Discrete}}
      \addlegendimage{empty legend}
     \addlegendentry{\hspace{-.8cm}\textbf{Ordinates ($N_Q$)}}

      \addplot+[red, mark options={black, scale=0.75},
      ] table[x=n_elements, y=time, col sep=comma] {tikz_Schaubilder/robustness_in_k/5_1_example_2.csv};
      \addlegendentry{2};
      \addplot+[green, mark options={black, scale=0.75},
      ] table[x=n_elements, y=time, col sep=comma] {tikz_Schaubilder/robustness_in_k/5_1_example_4.csv};
      \addlegendentry{4};
      \addplot+[blue, mark options={black, scale=0.75},
      ] table[x=n_elements, y=time, col sep=comma] {tikz_Schaubilder/robustness_in_k/5_1_example_8.csv};
      \addlegendentry{8};
      \addplot+[cyan, mark options={black, scale=0.75},
      ] table[x=n_elements, y=time, col sep=comma] {tikz_Schaubilder/robustness_in_k/5_1_example_16.csv};
      \addlegendentry{16};
      \addplot+[magenta, mark options={black, scale=0.75},
      ] table[x=n_elements, y=time, col sep=comma] {tikz_Schaubilder/robustness_in_k/5_1_example_32.csv};
      \addlegendentry{32};
      \addplot+[red, mark options={black, scale=0.75},
      ] table[x=n_elements, y=time, col sep=comma] {tikz_Schaubilder/robustness_in_k/5_1_example_64.csv};
      \addlegendentry{64};
      \addplot+[green, mark options={black, scale=0.75},
      ] table[x=n_elements, y=time, col sep=comma] {tikz_Schaubilder/robustness_in_k/5_1_example_128.csv};
      \addlegendentry{128};
      \addplot+[blue, mark options={black, scale=0.75},
      ] table[x=n_elements, y=time, col sep=comma] {tikz_Schaubilder/robustness_in_k/5_1_example_256.csv};
      \addlegendentry{256};
      \addplot+[cyan, mark options={black, scale=0.75},
      ] table[x=n_elements, y=time, col sep=comma] {tikz_Schaubilder/robustness_in_k/5_1_example_512.csv};
      \addlegendentry{512};

      \addplot[domain=25:6400, 
               samples=100, 
               color=red,
               dotted]
               {8e-4 * x * ln(x)};
      \addlegendentry{$\mathcal{O}(N\ln{N})$};
  
    \end{axis}
  \end{tikzpicture}
    \caption{\label{fig:CPU} CPU times for Algorithm~\ref{alg:Scheduler}, showing performance in generating sweeping permutations for a Voronoi mesh with varying numbers of elements and discrete ordinates.}
\end{figure}
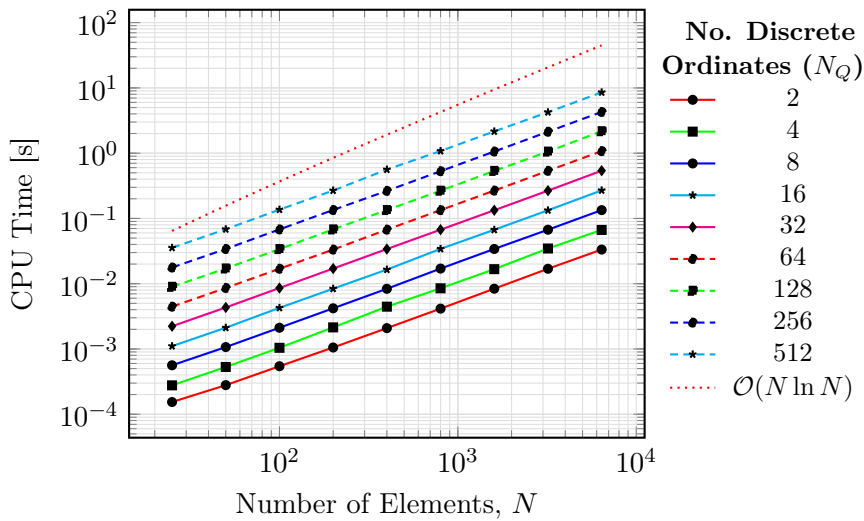
The results in Figure \ref{fig:CPU} confirm that the
algorithm's behaviour is as expected for varying $N_Q$ and $N$. The
geometry of the mesh does not affect the algorithm, making it
consistent and robust across varying values of $N_Q$ and $N$. The
results also demonstrate that CPU times scale equivalently across
cases where total degrees of freedom are comparable.

\subsection{Matrix Structure}

Applying Algorithm~\ref{alg:Scheduler} to
\eqref{eqn:full_matrix_system} results in a block lower-triangular
structure for $\vec{A}$. With the lowest order DG scheme, $p=0$, this
structure simplifies to a fully lower-triangular
system. Figure~\ref{fig:spies_a}
illustrates this for four discrete ordinates and one hundred spatial
elements. Results are analogous for higher values of $N_Q$ and $N$. We
also examine the structure of the scattering operator, $\vec{S}$,
which remains sparse after preconditioning but does not simplify to a
lower-triangular structure.

\begin{figure}[h!]
  \centering
  \subcaptionbox{Non-preconditioned matrix $\vec{A}$ with spy of $\vec{A}_3$}
  {\includegraphics[width=1\linewidth]{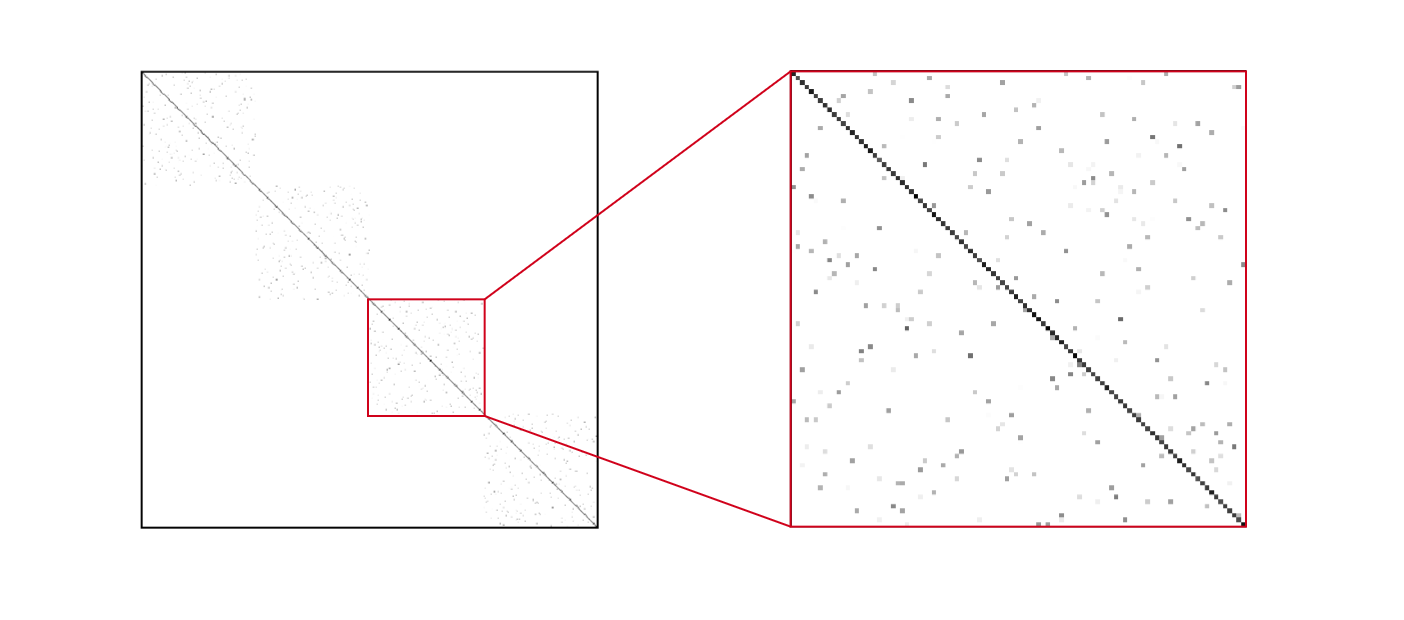}}%
  \hspace{4em}
  \subcaptionbox{Preconditioned matrix $\vec{P}\vec{A}$ with spy of $(\vec{P}\vec{A})_3$\label{fig:spy_a_swept}}
  {\includegraphics[width=1\linewidth]{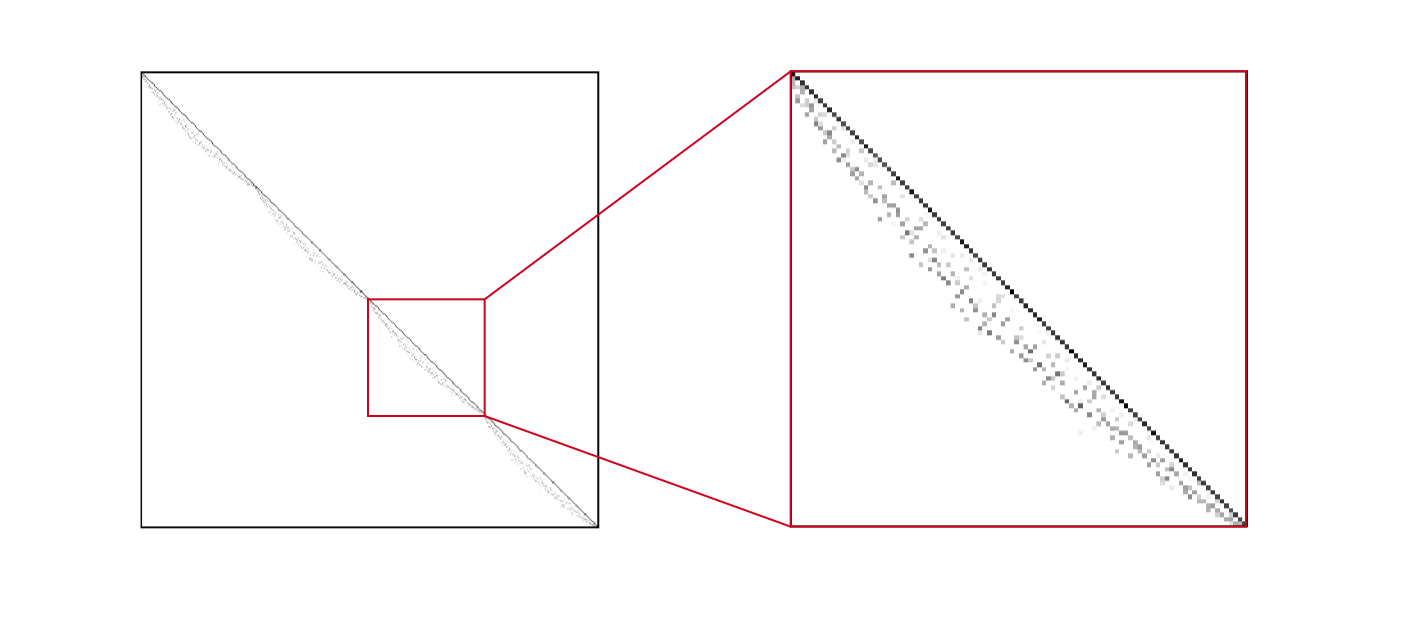}}
  \caption{Structure of the hyperbolic linear Boltzmann transport operator $\vec{A}$ before and after applying a sweeping permutation. The operator $\vec{A}$ represents the left-hand side in \eqref{eqn:bte}, while $\vec{P}\vec{A}$ denotes its permuted (swept) form. Subscripts indicate associations with the discrete ordinate direction $\vec{\omega}_k$.}
  \label{fig:spies_a}
\end{figure}

Figure~\ref{fig:spy_a_swept} demonstrates that the preconditioned
operator $(\vec{P}\vec{A})_k$ exhibits the expected lower-triangular
structure. This structure supports efficient upwind-downwind solution
transfer. By contrast, the scattering operator $\vec{S}$ shown in
Figure \ref{fig:spies_s} remains sparse but has a more complex block
structure after the sweep, with off-diagonal behaviour varying based on
the angles between ordinates. For example, $\vec{S}_{1,3}$ appears
anti-diagonal due to the opposing ordinates, whereas $\vec{S}_{2,3}$
has a circular structure due to perpendicularity. Higher ordinate
counts produce more intricate patterns, but crucially, the density and
structure do not seem to detract from the efficiency of the source
iteration algorithm.

\begin{figure}[h!]
  \centering
  \subcaptionbox{Non-preconditioned matrix $\vec{S}$ with spy of $\vec{S}_{2,3}$\label{fig:spyss_unswept}}
  {\includegraphics[width=1\linewidth]{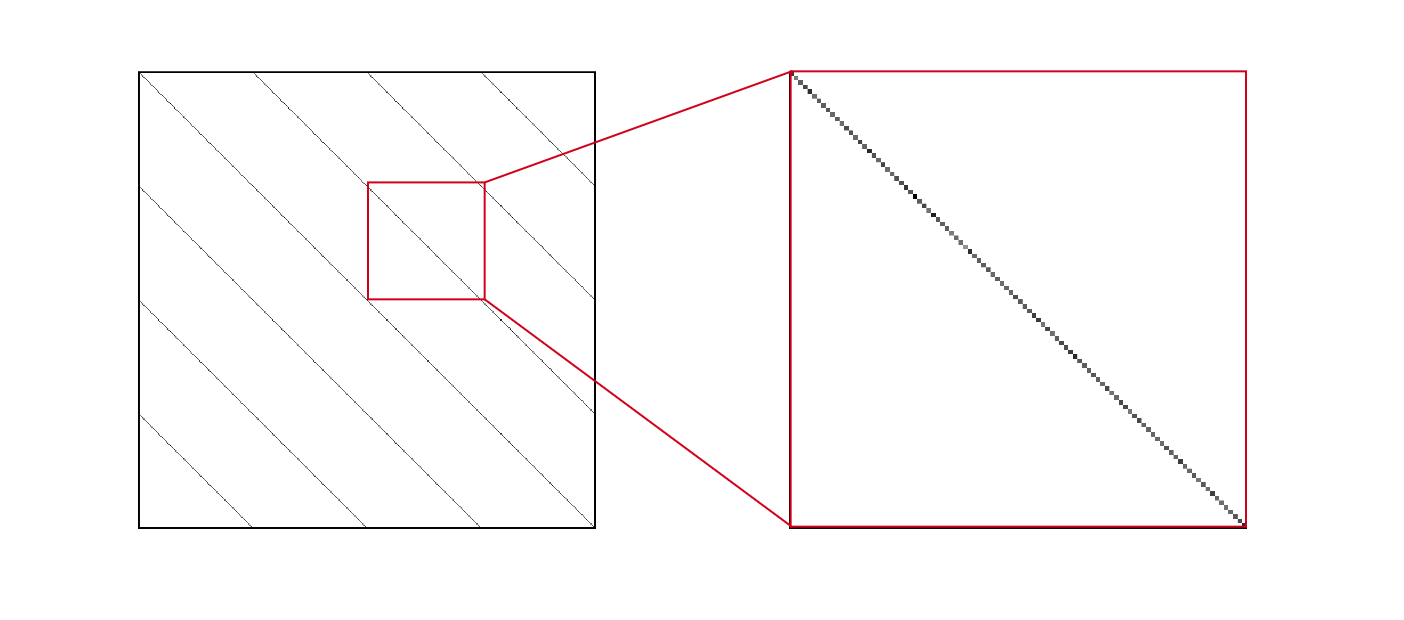}}%
  \hspace{4em}
  \subcaptionbox{Preconditioned matrix $\vec{P}\vec{S}$ with spy of $(\vec{P}\vec{S})_{2,3}$\label{fig:spy_s_swept}}
  {\includegraphics[width=1\linewidth]{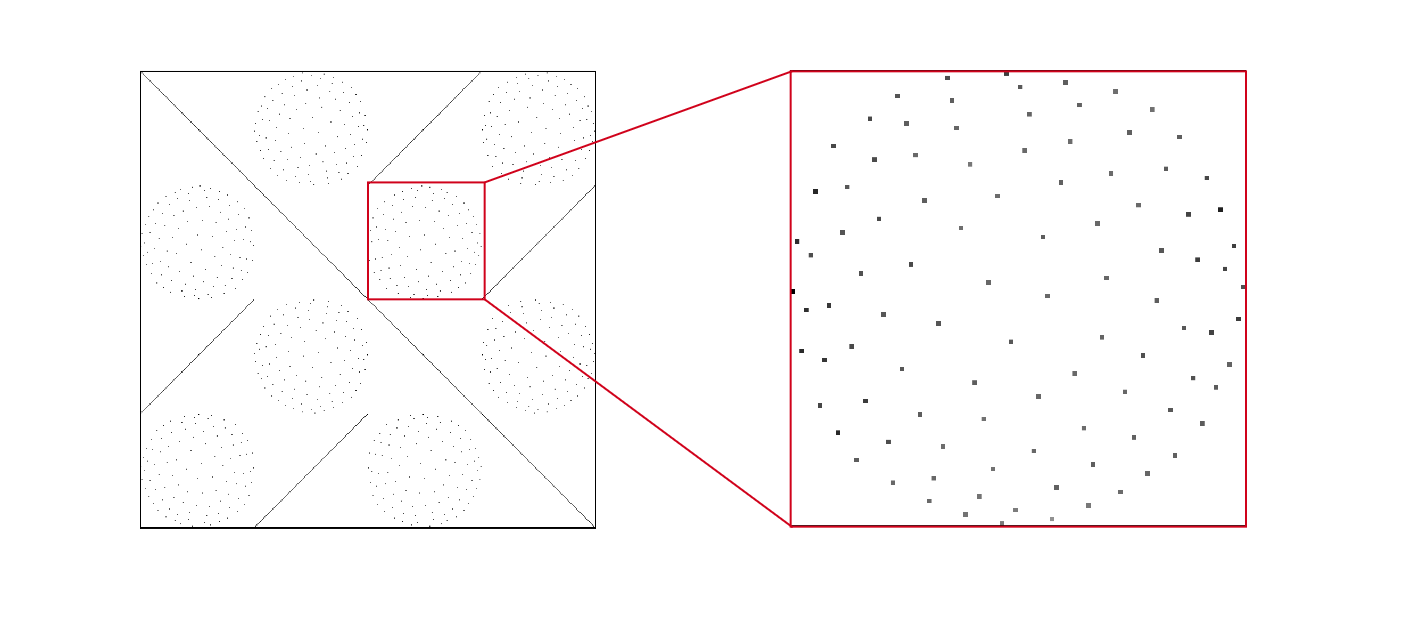}}
  \caption{Structure of the hyperbolic linear Boltzmann transport scattering operator $\vec{S}$ before and after a sweeping permutation. The operator $\vec{S}$ appears on the right-hand side of \eqref{eqn:bte}, with $\vec{P}\vec{S}$ representing its permuted (swept) form. Subscripts $(k, l)$ indicate scattering from ordinate $\vec{\omega}_k$ to $\vec{\omega}_l$.}
  \label{fig:spies_s}
\end{figure}


\subsection{Convergence Studies with the BTE}

We assess convergence from two perspectives: convergence with
respect to mesh refinement for each discrete ordinate direction
(ordinate-wise convergence) and convergence of the source
iteration itself.

For ordinate-wise convergence, we analyse the energy norm error, as
indicated in Lemma~\ref{lemma:error_control}. For the source iteration convergence, we use a Bochner-type norm over $L^2(\mathbb{S}^1;
L^2(\Omega))$, which accounts for errors across both angular and
spatial components and is given by the formula in Equation ~\eqref{eqn:bochner_norm}.

\begin{equation}
    \label{eqn:bochner_norm}
    \enorm{\vec\Psi_h - \vec{\Psi}_h^{(n)}}_{\Omega_{N_Q}} = \sum_{k=1}^{N_Q} w_k\enorm{\psi_k - \psi_{h,k}}
\end{equation}

Figure~\ref{fig:ordinate_wise_convergence} illustrates the error in the energy
norm as a function of the discrete ordinate direction
$\vec{\omega}_k$, for varying levels of spatial and angular
resolution. Figure~\ref{fig:ordinate_wise_convergence} shows the
asymptotic convergence of the error as a function of the maximal cell diameter for 128 discrete ordinates, which match the results presented in Lemma~\ref{lemma:error_control}.


\begin{figure}[h!]
   \centering
   \begin{tikzpicture}
    \begin{groupplot}[
        group style={
            group size=2 by 1,
            horizontal sep= 0.75cm,
            vertical sep=2cm,
            every plot/.style={ymin=0.06,ymax=0.16, axis line style={black, line width=0.5pt}},
        },
        enlargelimits=false,
    ]

    \nextgroupplot[
        ymode=log,
        xlabel={$\theta_k$},
        ylabel = {$\enorm{\psi_k - \psi_{h,k}}$},
        xmin=-0.01, 
        xmax=2*pi+0.01,
        grid=major,
        xtick={0, 0.5*pi, pi, 1.5*pi, 2*pi},
        xticklabels={$0$,$\frac{\pi}{2}$,$\pi$,$\frac{3\pi}{2}$,$2\pi$},
        legend style={at={(0.5, 0.1)},anchor=north,
        legend columns=4,
        fill=white,
        draw=black,
        anchor=center,
        align=center},
        style={line width=1.0pt}
    ]
    \addlegendimage{empty legend}
    \addlegendentry{\textbf{No. Elements:}\hspace{0.2cm}}

    \addplot+[only marks, mark options={black, mark=square*, scale=0.4}] 
    table[x=theta, y=25, col sep=comma] {tikz_Schaubilder/5_4_regions_convergence/128/5_4_energy_norm_angle_128.csv};
    \addlegendentry{25\hspace{0.2cm}};

    \addplot+[only marks, mark options={purple, mark=triangle*, scale=0.4}] 
    table[x=theta, y=50, col sep=comma] {tikz_Schaubilder/5_4_regions_convergence/128/5_4_energy_norm_angle_128.csv};
    \addlegendentry{50\hspace{0.2cm}};

    \addplot+[only marks, mark options={pink, scale=0.4}] 
    table[x=theta, y=100, col sep=comma] {tikz_Schaubilder/5_4_regions_convergence/128/5_4_energy_norm_angle_128.csv};
    \addlegendentry{100\hspace{0.2cm}};




    \addplot+[only marks, mark=+, mark options={red,scale=2, line width=2pt, rotate=45}] table {
        0.05 0.146662
        0.05 0.123099
        0.05 0.103643
        };

    \addplot+[only marks, mark=square, mark options={orange,scale=1, line width=2pt, rotate=45}] table {
        4.00607 0.154564
        4.00607 0.1303472
        4.00607 0.110541
        };

    \addplot+[only marks, mark options={blue,scale=1, line width=2pt, rotate=45}] table {
        4.84056 0.139330
        4.84056 0.115056
        4.84056 0.100366
        };

    \nextgroupplot[ 
        ymode=log,
        xmode=log,
        grid=major,
        xtick={0.15849, 0.19953, 0.25119}, 
        xticklabels={$10^{-0.8}$, $10^{-0.7}$, $10^{-0.6}$},
        xlabel={$h$},
        yticklabel=\empty,
        legend style={at={(0.5, 0.1)},anchor=north,
        legend columns=4,
        fill=white,
        draw=black,
        anchor=center,
        align=center},
        style={line width=1.0pt}
    ]
    \addlegendimage{empty legend}
    \addlegendentry{\textbf{$\theta_k\approx$}\hspace{0.2cm}}

    \addplot+[black, mark=+, mark options={red,scale=2, line width=2pt, rotate=45}] 
    table[x=h, y=sub_norm, col sep=comma] {tikz_Schaubilder/5_4_regions_convergence/sub_norms/128_0.csv};
    \addlegendentry{0.03\hspace{0.2cm}};

    \addplot+[black, mark=square, mark options={orange,scale=1, line width=2pt, rotate=45}] 
    table[x=h, y=sub_norm, col sep=comma] {tikz_Schaubilder/5_4_regions_convergence/sub_norms/128_81.csv};
    \addlegendentry{4.00\hspace{0.2cm}};

    \addplot+[black, mark options={blue,scale=1, line width=2pt, rotate=45}] 
    table[x=h, y=sub_norm, col sep=comma] {tikz_Schaubilder/5_4_regions_convergence/sub_norms/128_98.csv};
    \addlegendentry{4.84\hspace{0.2cm}};

    \addplot[domain=0.156:0.298, samples=100, color=black, dotted, very thick]{0.23
    * x^(0.5)};
    \node at (axis cs:0.22,0.1) {$\mathcal{O}(h^{\frac{1}{2}})$};

    \end{groupplot}

\end{tikzpicture}
   \caption{
    \label{fig:ordinate_wise_convergence}
A comparison of error metrics. (Left) Error as a function of angle for three spatial mesh resolutions, with specific angles marked by a circle, cross, and diamond. (Right) Error as a function of spatial mesh diameter at the marked angles, illustrating the relationship between angular error and spatial resolution.
    }
\end{figure}
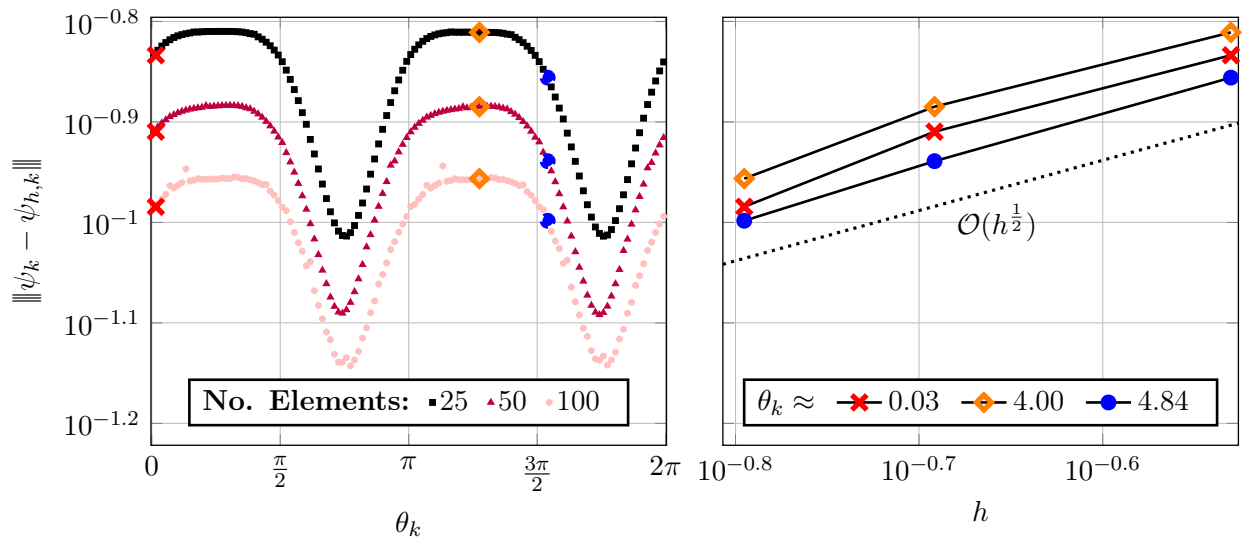

Finally in Figure~\ref{fig:section_5_4_iterations}, we assess the convergence of the source iteration itself,
comparing the iterates $\vec{\Psi}_h^{(n)}$ to the solution $\vec\Psi_h$ for various scattering
ratios, $c = \sup_{\vec{x}\in\Omega}
\frac{\Sigma_s(\vec{x})}{\Sigma_t(\vec{x})}$, to illustrate the effect
of scattering.

Results are shown for $c=0.7, 0.95, 0.999$ and varying
degrees of freedom with 64 ordinates. We also examine the spectral
radius, noting that source iteration convergence depends on $\rho(\vec
A^{-1}\vec S)$ due to its Richardson iteration-like structure.

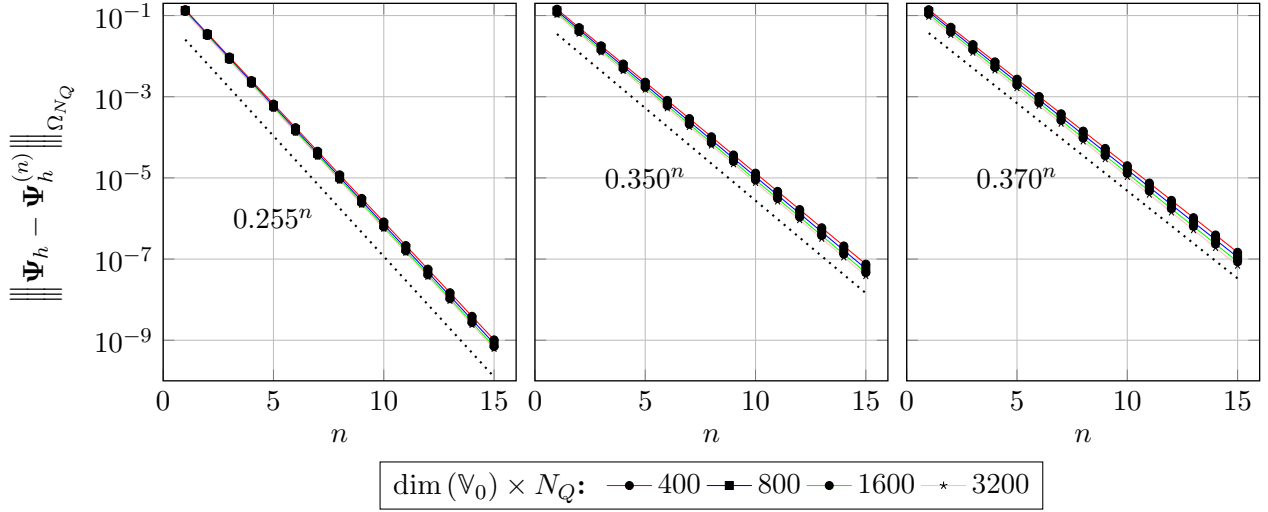
\begin{figure}
    \centering
    \begin{tikzpicture}
    \begin{groupplot}[
        group style={
            group size=3 by 1,
            horizontal sep= 0.25cm, 
            vertical sep=1cm,
            every plot/.style={
            ymin=10e-11,ymax=0.2, 
            axis line style={black, line width=0.5pt},
            xmin=0, xmax=16,
            },
        },
        enlargelimits=false,
    ]
 
      \nextgroupplot[
        width=0.38\textwidth,
        height=0.4\textwidth,
        xlabel={$n$},
        ylabel = {$\enorm{\vec\Psi_h - \vec\Psi_h^{(n)}}_{\Omega_{N_Q}}$},
        ymode=log,
        grid=major,
        legend style={
            at={($(0,0)+(1cm,1cm)$)},
            legend columns=5,
            fill=none,
            draw=black,
            anchor=center,
            align=center},
        legend to name=legend,
    ]
    \coordinate (c1) at (rel axis cs:0,1);
    \addlegendimage{empty legend}
    \addlegendentry{\textbf{$\dim{(\fes_0)\times N_Q}$:}\hspace{0.2cm}}
    
    \addplot+[red, mark options={black, scale=0.75}] 
    table[x=iterates, y=25, col sep=comma] {tikz_Schaubilder/5_4_iteration_convergence/0_7/5_4_convergence_0_7_64.csv};
    \addlegendentry{400};
    
    \addplot+[blue, mark options={black, scale=0.75}] 
    table[x=iterates, y=50, col sep=comma] {tikz_Schaubilder/5_4_iteration_convergence/0_7/5_4_convergence_0_7_64.csv};
    \addlegendentry{800};
    
    \addplot+[green, mark options={black, scale=0.75}] 
    table[x=iterates, y=100, col sep=comma] {tikz_Schaubilder/5_4_iteration_convergence/0_7/5_4_convergence_0_7_64.csv};
    \addlegendentry{1600};
    
    \addplot+[pink, mark options={black, scale=0.75}] 
    table[x=iterates, y=200, col sep=comma] {tikz_Schaubilder/5_4_iteration_convergence/0_7/5_4_convergence_0_7_64.csv};
    \addlegendentry{3200};
    
    \addplot[domain=1:15, samples=100, color=black, thick, dotted]{0.1 * 0.255^x};
    \node at (axis cs:5,10e-7) {$0.255^n$};
    
    \nextgroupplot[
        width=0.38\textwidth,
        height=0.4\textwidth,
        xlabel={$n$},
        ymode=log,
        yticklabels={},
        grid=major,
        legend style={at={(1.05,1)}, anchor=north west},
    ]
    
    \addplot+[red, mark options={black, scale=0.75}] 
    table[x=iterates, y=25, col sep=comma] {tikz_Schaubilder/5_4_iteration_convergence/0_95/5_4_convergence_0_95_64.csv};
    
    \addplot+[blue, mark options={black, scale=0.75}] 
    table[x=iterates, y=50, col sep=comma] {tikz_Schaubilder/5_4_iteration_convergence/0_95/5_4_convergence_0_95_64.csv};
    
    \addplot+[green, mark options={black, scale=0.75}] 
    table[x=iterates, y=100, col sep=comma] {tikz_Schaubilder/5_4_iteration_convergence/0_95/5_4_convergence_0_95_64.csv};
    
    \addplot+[pink, mark options={black, scale=0.75}] 
    table[x=iterates, y=200, col sep=comma] {tikz_Schaubilder/5_4_iteration_convergence/0_95/5_4_convergence_0_95_64.csv};
    
    \addplot[domain=1:15, samples=100, color=black, thick, dotted]{0.1 * 0.35^x};
    \node at (axis cs:5,10e-6) {$0.350^n$};

    \nextgroupplot[
        width=0.38\textwidth,
        height=0.4\textwidth,
        xlabel={$n$},
        yticklabels={},
        ymode=log,
        grid=major,
        legend style={at={(1.05,1)}, anchor=north west},
    ]
    \coordinate (c2) at (rel axis cs:1,1);
    
    \addplot+[red, mark options={black, scale=0.75}] 
    table[x=iterates, y=25, col sep=comma] {tikz_Schaubilder/5_4_iteration_convergence/0_999/5_4_convergence_0_999_64.csv};
    
    \addplot+[blue, mark options={black, scale=0.75}] 
    table[x=iterates, y=50, col sep=comma] {tikz_Schaubilder/5_4_iteration_convergence/0_999/5_4_convergence_0_999_64.csv};
    
    \addplot+[green, mark options={black, scale=0.75}] 
    table[x=iterates, y=100, col sep=comma] {tikz_Schaubilder/5_4_iteration_convergence/0_999/5_4_convergence_0_999_64.csv};
    
    \addplot+[pink, mark options={black, scale=0.75}] 
    table[x=iterates, y=200, col sep=comma] {tikz_Schaubilder/5_4_iteration_convergence/0_999/5_4_convergence_0_999_64.csv};
    
    \addplot[domain=1:15, samples=100, color=black, thick, dotted]{0.1 * 0.37^x};
    \node at (axis cs:5,10e-6) {$0.370^n$};

    \end{groupplot}

    \coordinate (c3) at ($(c1)!.5!(c2)$);
    \node[below] at (c3 |- current bounding box.south)
      {\pgfplotslegendfromname{legend}};
    
\end{tikzpicture}
    \caption{Convergence of iterates $\vec\Psi_h^{(n)}$ from \eqref{eqn:swept_full_system} to the numerical solution $\vec\Psi_h$ of \eqref{eqn:full_matrix_system}, shown for scattering ratios $c = 0.7$, $0.95$, and $0.999$. Overlaid trend lines indicate expected convergence rates based on the spectral radius of the iteration matrix.}
      \label{fig:section_5_4_iterations}
\end{figure}
The spectral radius is lower than the scattering ratio in all
cases. As $c\to 1$, convergence slows but remains adequate for
practical purposes. Preconditioning through sweeping notably
accelerates convergence and reduces computational complexity through
the use of a lower triangular solver.

\subsection{Complex Geometries}
\label{section:numerics:complex_geometries}

Finally, we consider a more complex geometry, representative of a
reactor core as presented in \cite{10.1007/978-3-642-15414-0_21}, to
demonstrate mesh generation on a non-trivial
domain. Figure~\ref{fig:reactor_core} shows the Voronoi tessellation
of the domain, with each region colour-coded by core material.

\begin{figure}[h!]
  \centering
  \subcaptionbox{Voronoi tessellation of reactor core, with core material colour-coded in the legend.\label{fig:reactor_core}}[0.44\textwidth]
  {\includegraphics[width=0.44\textwidth]{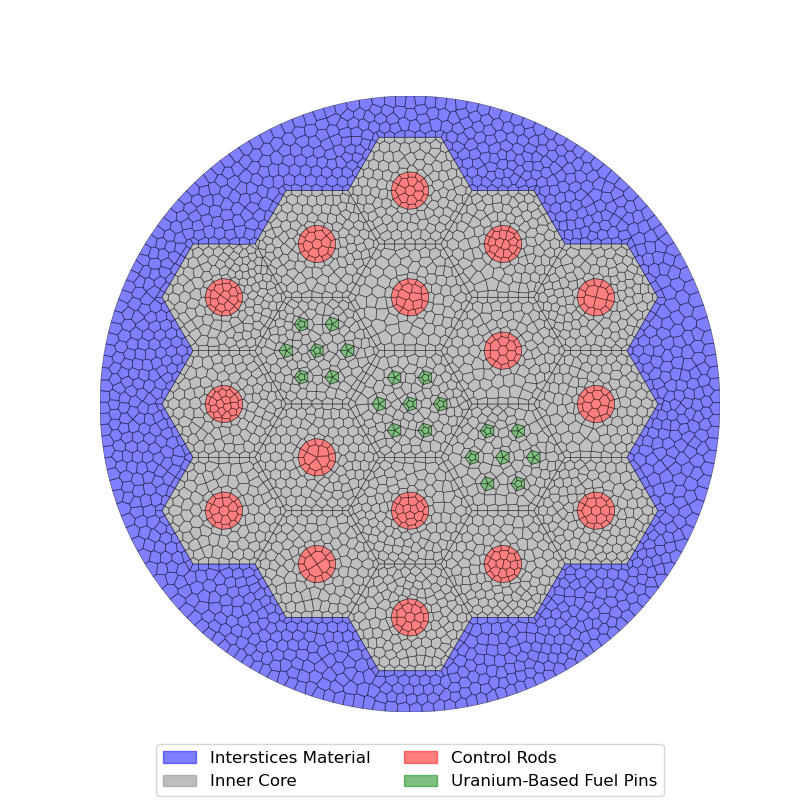}}
  \hfill
  \subcaptionbox{Numerical scalar flux $\phi_h$ for Equation~\eqref{eqn:swept_full_system}, calculated over the mesh in Figure~\ref{fig:reactor_core}.\label{fig:full_simulation}}[0.54\textwidth]
  {\includegraphics[width=0.54\textwidth]{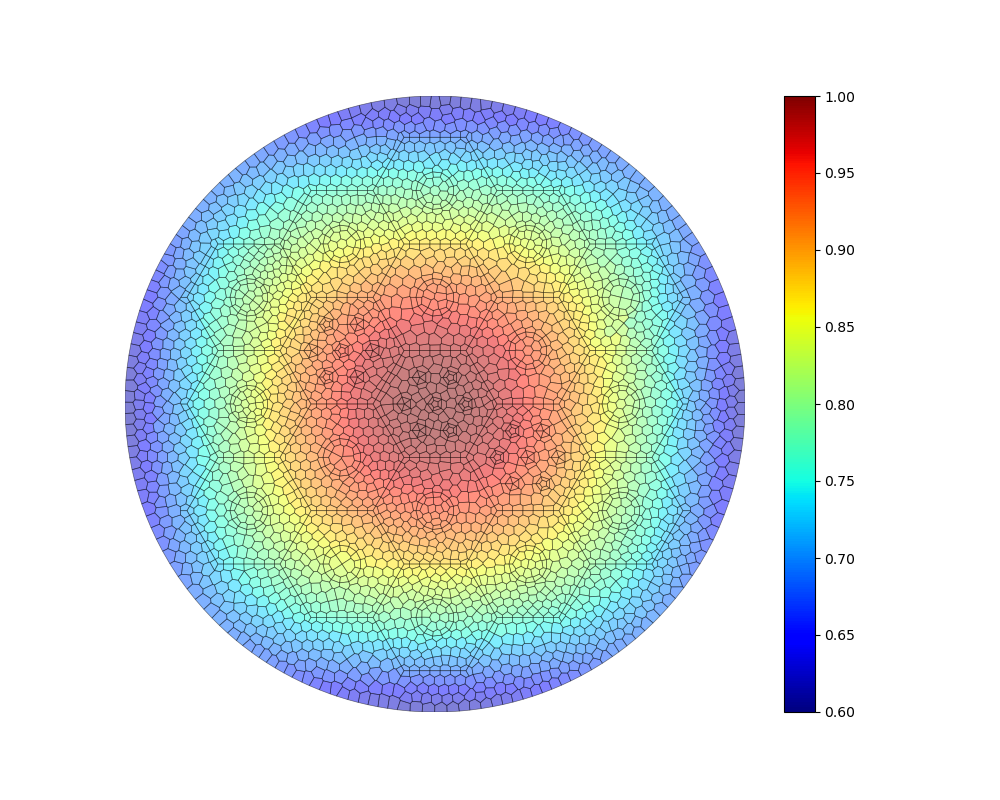}\label{fig:full_simulation}}
  \caption{Voronoi mesh of a reactor core with the corresponding numerical scalar flux solution to \eqref{eqn:swept_full_system}, illustrating radiative intensity across the reactor domain.}
  \label{fig:reactor_images}
\end{figure}


The Voronoi tessellation effectively captures complex domain geometry
and is compatible with sweeping techniques. PolyMesher
\cite{PolyMesher} was used to generate the mesh, with fixed points for
regions requiring higher resolution.

For this example, we simulate the BTE solution on the mesh shown in
Figure~\ref{fig:reactor_core}, modifying the macroscopic and
scattering cross-sections to yield a uniform scattering ratio across
the domain. The scalar flux solution, defined as
\begin{equation}
    \label{eqn:numerical_scalar_flux}
    \phi_h(\vec{x}) = \sum_{k=1}^{N_Q} w_k\psi_{h,k}(\vec{x}) \approx \frac{1}{|\mathbb{S}^{d-1}|}\int_{\mathbb{S}^{d-1}}\psi(\vec{x}, \vec{\omega})\,d\vec{\omega} = \phi(\vec{x}),
\end{equation}
is shown in Figure~\ref{fig:full_simulation}. The solution exhibits
expected behavior based on Equation~\eqref{eqn:numerics_solution} and
demonstrates the viability of Voronoi meshes and sweeping algorithms
for both complex and practical geometries.

\section{Conclusion}
\label{section:summary}

In this paper, we developed and analysed mesh sweeping methods within
the framework of transport equations, with a particular focus on
Voronoi tessellations as a robust solution for achieving cycle-free
sweeps across multiple directions and dimensions. Our study showed
that Voronoi meshes inherently support omnidirectional acyclic sweeps,
a property that enables efficient solutions for complex transport
problems, including a simplified form of the linear Boltzmann
transport equation.

Key contributions of this work include the theoretical foundation
establishing Voronoi meshes as omnidirectionally acyclic and the
introduction of a scheduler that guarantees topologically sorted,
sweep-compatible meshes. Numerical results confirmed that this
approach achieves near-optimal convergence rates, meeting or exceeding
the performance of existing methods while eliminating the need for
deadlock resolution, an advancement that simplifies implementation and
enhances computational efficiency. This makes the proposed framework
particularly well-suited for polytopal DG methods.

Our findings open several avenues for future research, including
integration with adaptive mesh refinement techniques to dynamically
balance computational cost and accuracy in domains with localised
features. Additionally, investigating sweeping methods on alternative
mesh types that share or approximate the omnidirectional acyclic
properties of Voronoi meshes could broaden the impact of this
approach, making it applicable to a wider range of meshing
strategies. This framework has the potential to facilitate
applications in more diverse fields, including charged particle
transport, where the transport direction varies and characteristics
follow curved paths rather than straight lines.

\section{Data Availability}

The code to generate all the data in this paper is included in the
Zenodo archive linked to this paper \cite{zenodo}.

\section{Acknowledgements}

This project was instigated at an Integrative Think Tank run by the
Statistical Applied Mathematics at Bath (SAMBa) EPSRC Centre for
Doctoral Training (CDT) at the University of Bath. ME \& HL are
supported by scholarships from SAMBa under the project
EP/S022945/1. ME is also partially funded by the French Alternative
Energies and Atomic Energy Commission (CEA). TP received support from
the EPSRC programme grant EP/W026899/1, the Leverhulme RPG-2021-238
and EPSRC grant EP/X030067/1. All of this support is gratefully
acknowledged.

\bibliography{Literatur.bib} 

\end{document}